\documentclass[11pt,reqno]{amsart}

\usepackage{url}
\usepackage{amssymb}
\usepackage{amscd}
\usepackage{amsfonts}
\usepackage{amsmath}
\usepackage{eepic}
\usepackage{gastex}
\usepackage{amsthm}
\usepackage{amsgen}
\usepackage[dvips]{color}
\usepackage{eucal}
\usepackage[square]{natbib}

\newcommand{\Rm}{Rees matrix}
\newcommand{\Rms}{Rees matrix \sgp}
\newcommand{\sgp}{semi\-group}
\newcommand{\sgps}{semi\-groups}

\newcommand{\fb}{finitely based}

\newcommand{\nfb}{non\-finitely based}

\newcommand{\fg}{finitely generated}
\newcommand{\is}{involution semi\-group}
\newcommand{\iss}{involution semi\-groups}

\def\FI{\ensuremath{\mathcal{FI}}}
\def\malcev{\mathop{\hbox{$\bigcirc$\kern-9.5pt\raise1pt\hbox{\scriptsize$m$}\kern1.5pt}}}

\newcommand{\TSL}{\ensuremath{\mathcal{T\kern-.5pt S\kern-.5pt L}}}

\DeclareMathOperator{\var}{\mathsf{var}}

\newcommand{\wire}[2]{\begin{picture}(13,3)\gasset{AHnb=0,linewidth=0.3}\drawline(3,1)(8,1)\put(0,0){$#1$}\put(9,0){$#2$}\end{picture}}
\newcommand{\wirei}[2]{\begin{picture}(13,3)\gasset{AHnb=0,linewidth=0.3}\drawline(4,1)(8,1)\put(0,0){$#1$}\put(9,0){$#2$}\end{picture}}
\newcommand{\wireii}[2]{\begin{picture}(16,3)\gasset{AHnb=0,linewidth=0.3}\drawline(4,1)(8,1)\put(0,0){$#1$}\put(9,0){$#2$}\end{picture}}
\newcommand{\wires}[2]{\begin{picture}(12,3)\gasset{AHnb=0,linewidth=0.3}\drawline(3,1)(8,1)\put(0,0){$#1$}\put(9,0){$#2$}\end{picture}}

\theoremstyle{plain}
\newtheorem{theorem}{Theorem}
\newtheorem{lemma}[theorem]{Lemma}
\newtheorem{proposition}[theorem]{Proposition}

\theoremstyle{remark}
\newtheorem{remark}{Remark}

\title[The Finite Basis Problem for Kauffman Monoids]{The Finite Basis Problem\\ for Kauffman Monoids}
\author{K. Auinger}
\address{(K. Auinger) Fakult\"at f\"ur Mathematik, Universit\"at Wien, Oskar-Morgen\-stern-Platz 1,  A-1090 Wien, Austria}
\email{karl.auinger@univie.ac.at}

\author{Yuzhu Chen, Xun Hu, Yanfeng Luo}
\address{(Yuzhu Chen, Xun Hu, Yanfeng Luo) Department of Mathematics and Statistics, Key Laboratory of Applied Mathematics and Complex Systems,
Lan\-zhou University, Lanzhou, Gansu, 730000, China} \email{luoyf@lzu.edu.cn}

\address{(Xun Hu) Department of Mathematics and Statistics, Chongqing Technology and Business University,
Chongqing, 400033, China}

\author{M. V. Volkov}
\address{(M. V. Volkov) Institute of Mathematics and Computer Science, Ural Federal University,
Lenina 51, 620000 Ekaterinburg, Russia} \email{mikhail.volkov@usu.ru}

\begin{document}

\begin{abstract}
We prove a sufficient condition under which a semigroup admits no finite identity basis. As an application, it is shown that the identities
of the Kauffman monoid $\mathcal{K}_n$ are nonfinitely based for each $n\ge 3$. This result holds also for the case when $\mathcal{K}_n$ is
considered as an involution semigroup under either of its natural involutions.
\end{abstract}

\maketitle

\section*{Introduction}

\citet{TL71}, motivated by some graph-theoretic problems in
statistical mechanics, introduced what is now called the
\emph{Temperley--Lieb algebras}. These are associative linear
algebras with 1 over a commutative ring $R$. Given an integer
$n\ge 2$ and a scalar $\delta\in R$, the Temperley--Lieb algebra
$\mathcal{TL}_n(\delta)$ is generated by elements
$h_1,\dots,h_{n-1}$ subject to the relations
\begin{align}
&h_{i}h_{j}=h_{j}h_{i}    &&\text{if } |i-j|\ge 2,\ i,j=1,\dots,n-1;\label{eq:TL1}\\
&h_{i}h_{j}h_{i}=h_{i}    &&\text{if } |i-j|=1,\ i,j=1,\dots,n-1;\label{eq:TL2}\\
&h_{i}h_{i}=\delta h_{i}  &&\text{for each } i=1,\dots,n-1.\label{eq:TL3}
\end{align}
The relations \eqref{eq:TL1}--\eqref{eq:TL3} are `multiplicative', i.e., they do not involve addition. This observation suggests
introducing a monoid whose monoid algebra over $R$ could be identified with $\mathcal{TL}_n(\delta)$. A tiny obstacle is the presence of
the scalar $\delta$ in \eqref{eq:TL3}, but it can be bypassed by adding a new generator $c$ that imitates $\delta$. This way one comes to
the monoid $\mathcal{K}_n$ with $n$ generators $c,h_1,\dots,h_{n-1}$ subject to the relations \eqref{eq:TL1}, \eqref{eq:TL2}, and the
relations
\begin{align}
&h_{i}h_{i}=ch_{i}=h_{i}c   &&\text{for each } i=1,\dots,n-1,\label{eq:TL4}
\end{align}
which both mimic \eqref{eq:TL3} and mean that $c$ behaves like a
scalar. The monoids $\mathcal{K}_n$ are called the \emph{Kauffman
monoids}\footnote{The name was suggested by \citet{BDP02}; in the
literature one also meets the name \emph{Temperley--Lieb--Kauffman
monoids} \citep[see, e.g.,][]{BL05}. Kauffman himself used the
term \emph{connection monoids}.} after \citet{Ka90} who
independently invented these monoids as geometric objects. It
turns out that Kauffman monoids play a major role in several
`fashionable' parts of mathematics such as knot theory,
low-dimensional topology, topological quantum field theory,
quantum groups etc. As algebraic objects, these monoids belong to
the family of so-called diagram or Brauer-type monoids that
originally arose in representation theory and gained much
attention recently among semigroup theorists. In particular, the
first-named author (solo and with collaborators) has considered
universal-algebraic aspects of some monoids from this family such
as the finite basis problem for their identities or the
identification of the pseudovarieties generated by certain
series of such monoids \citep[see, e.g.,][]{Au14,ADV12b}. In the
present paper we follow this line of research and investigate the
finite basis problem for the identities holding in Kauffman
monoids.

Whilst it is not clear whether or not a study of the identities of Kauffman monoids may be of any use for any of their non-algebraic
applications, such a study constitutes an interesting challenge from the algebraic viewpoint since---in contrast to other types of diagram
monoids---Kauffman monoids are infinite. We recall that there exist several powerful methods to attack the finite basis problem for
\textbf{finite} semigroups (see the survey \citep{Vo01} for an overview), but, to the best of our knowledge, so far the problem has been
solved for only one natural family of concrete infinite semigroups that contains semigroups satisfying a nontrivial identity, namely, for
non-cyclic one-relator semigroups and monoids \citep{Sh89}. Here we prove that, for each $n\ge 3$, the identities of the monoid
$\mathcal{K}_n$ are not finitely based. The monoid $\mathcal{K}_2$ is commutative, and thus, its identities are finitely based.

The paper is structured as follows. In Section~\ref{sec:wire} we present geometric definitions for some classes of diagram monoids
including Kauffman monoids and so-called Jones monoids. We also summarize properties of Kauffman and Jones monoids which are essential for
the proofs of our main results. Section~\ref{sec:nfb} contains a new sufficient condition under which a semigroup admits no finite identity
basis. In Section~\ref{sec:applications} this condition is applied to the monoid $\mathcal{K}_n$ with $n\ge3$, thus showing that the
identities of $\mathcal{K}_n$ are \nfb; we also observe that the same result holds also for the case when $\mathcal{K}_n$ is considered as
an \is\ under either of its natural involutions. Besides that, we demonstrate a further application of our sufficient condition.

The fact that the identities of $\mathcal{K}_n$ with $n\ge 4$ are
\nfb\ was announced by the last-named author in his invited
lecture at the 3rd Novi Sad Algebraic Conference held in August
2009. Slides of this lecture\footnote{See
\url{http://csseminar.kadm.usu.ru/SLIDES/nsac2009/volkov_nsac.pdf}.}
included an outline of the proof for $n\ge 4$ as well as an
explicit mentioning that the case $n=3$ was left open. This case
has been recently analyzed by the first-named author and, independently and by completely
different methods, by the three
`middle-named' authors of the present paper: it turns out that
also the identities of $\mathcal{K}_3$ are \nfb. Naturally, the
authors have decided to join their results into a single article,
and so the present paper has been originated. The unified proof
presented here is based on the approach by the first-named and the
last-named authors. The alternative approach by the three
`middle-named' is of a syntactic flavor; it also has some further
applications and will be published in a separate paper.

\section{Diagrams and their multiplication}
\label{sec:wire}

The primary aim of this section is to present a geometric
definition for a series of diagram monoids which we call the
\emph{wire monoids} $\mathcal{W}_n$, $n\ge 2$. Each Kauffman
monoid $\mathcal{K}_n$ can be identified with a natural submonoid
of the corresponding wire monoid $\mathcal{W}_n$ so that a
geometric definition for the Kauffman monoids appears as a special
case. The reader should be advised that even though this geometric
definition certainly clarifies the nature of Kauffman monoids and
is crucial to their connections to various parts of mathematics,
knowing it is not really necessary for understanding the proofs in
the present paper. Therefore those readers who are mainly
interested in the finite basis problem for $\mathcal{K}_n$ may
skip the `geometric part' of this section and rely on the
definition of Kauffman monoids in terms of generators and
relations as stated in the introduction and on a similar
definition of Jones monoids given at the end of the section.

We fix an integer $n\ge 2$ and define the wire monoid $\mathcal{W}_n$. Let
$$[n]:=\{1,\dots,n\},\quad [n]':=\{1',\dots,n'\}$$
be two disjoint copies of the set of the first $n$ positive integers. The base set of $\mathcal{W}_n$ is the set of all pairs $(\pi;d)$
where $\pi$ is a partition of the $2n$-element set $[n]\cup [n]'$ into 2-element blocks and $d$ is a non-negative integer referred to as
the \emph{number of circles}. Such a pair is represented by a \emph{wire diagram} as shown in Figure~\ref{fig:diagram}.
\begin{figure}[hb]
\centering
\unitlength=.7mm
\begin{picture}(45,80)(0,10)
\drawrect[linewidth=.3](5,12,40,88)
\gasset{AHnb=0,linewidth=.5}
\multiput(5,18)(0,8){9}{\line(-1,0){2}}
\multiput(40,18)(0,8){9}{\line(1,0){2}}
\drawline(5,18)(40,50)
\drawline(5,58)(40,82)
\drawline(5,74)(40,74)
\drawcurve(5,26)(15,33)(5,42)
\drawcurve(5,34)(15,42)(5,50)
\drawcurve(5,66)(15,74)(5,82)
\drawcurve(40,18)(35,22)(40,26)
\drawcurve(40,34)(35,38)(40,42)
\drawcurve(40,58)(35,62)(40,66)
\put(-1,16){$\scriptstyle 1$}
\put(-1,24){$\scriptstyle 2$}
\put(-1,32){$\scriptstyle 3$}
\put(-1,40){$\scriptstyle 4$}
\put(-1,48){$\scriptstyle 5$}
\put(-1,56){$\scriptstyle 6$}
\put(-1,64){$\scriptstyle 7$}
\put(-1,72){$\scriptstyle 8$}
\put(-1,80){$\scriptstyle 9$}
\put(43,16){$\scriptstyle 1'$}
\put(43,24){$\scriptstyle 2'$}
\put(43,32){$\scriptstyle 3'$}
\put(43,40){$\scriptstyle 4'$}
\put(43,48){$\scriptstyle 5'$}
\put(43,56){$\scriptstyle 6'$}
\put(43,64){$\scriptstyle 7'$}
\put(43,72){$\scriptstyle 8'$}
\put(43,80){$\scriptstyle 9'$}
\drawcircle(22.5,50,3)
\drawcircle(22.5,50,6)
\drawcircle(22.5,50,9)
\end{picture}
\caption{Wire diagram representing an element of $\mathcal{W}_9$}\label{fig:diagram}
\end{figure}
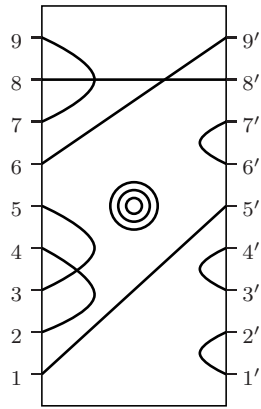
We draw a rectangular `chip' with $2n$ `pins' and represent the elements of $[n]$ by pins on the left hand side of the chip (\emph{left
pins}) while the elements of $[n]'$ are represented by pins on the right hand side of the chip (\emph{right pins}). Usually we omit the
numbers $1,2,\dots$ in our illustrations. Now, for $(\pi;d)\in \mathcal{W}_n$, we represent the number $d$ by $d$ closed curves (`circles')
drawn somewhere within the chip and each block of the partition $\pi$ is represented by a line referred to as a \emph{wire}. Thus, each
wire connects two pins; it is called an $\ell$-\emph{wire} if it connects two left pins, an $r$-\emph{wire} if it connects two right pins,
and a $t$-\emph{wire} if it connects a left pin with a right pin. The wire diagram in Figure~\ref{fig:diagram} corresponds to the pair
$$\Bigl(\bigr\{\{1,5'\},\{2,4\},\{3,5\},\{6,9'\},\{7,9\},\{8,8'\},\{1',2'\},\{3',4'\},\{6',7'\}\bigr\};\,3\Bigr).$$

Next we explain the multiplication in $\mathcal{W}_n$. Pictorially, in order to multiply two chips, we `shortcut' the right pins of the
first chip with the corresponding left pins of the second chip. Thus we obtain a new chip whose left (respectively, right) pins are the
left (respectively, right) pins of the first (respectively, second) chip and whose wires are sequences of consecutive wires of the factors,
see Figure~\ref{fig:multiplication}. All circles of the factors are inherited by the product; in addition, some extra circles may arise
from $r$-wires of the first chip combined with $\ell$-wires of the second chip.
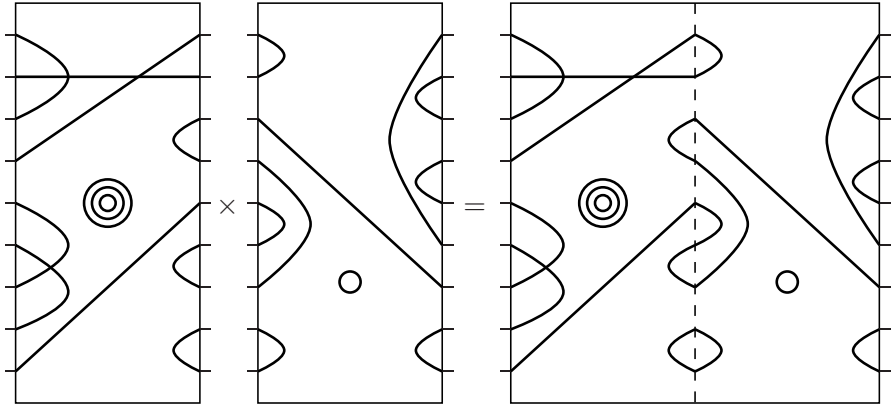
\begin{figure}[hb]
\centering
\unitlength=.7mm
\begin{picture}(170,80)(0,10)
\drawrect[linewidth=.3](5,12,40,88)
\gasset{AHnb=0,linewidth=.5}
\multiput(5,18)(0,8){9}{\line(-1,0){2}}
\multiput(40,18)(0,8){9}{\line(1,0){2}}
\drawline(5,18)(40,50)
\drawline(5,58)(40,82)
\drawline(5,74)(40,74)
\drawcurve(5,26)(15,33)(5,42)
\drawcurve(5,34)(15,42)(5,50)
\drawcurve(5,66)(15,74)(5,82)
\drawcurve(40,18)(35,22)(40,26)
\drawcurve(40,34)(35,38)(40,42)
\drawcurve(40,58)(35,62)(40,66)
\drawcircle(22.5,50,3)
\drawcircle(22.5,50,6)
\drawcircle(22.5,50,9)
\put(43,48){$\times$}
\put(46,0){\begin{picture}(50,80)
\drawrect[linewidth=.3](5,12,40,88)
\multiput(5,18)(0,8){9}{\line(-1,0){2}}
\multiput(40,18)(0,8){9}{\line(1,0){2}}
\drawline(5,66)(40,34)
\drawcurve(5,18)(10,22)(5,26)
\drawcurve(5,34)(15,46)(5,58)
\drawcurve(5,42)(10,46)(5,50)
\drawcurve(5,74)(10,78)(5,82)
\drawcurve(40,18)(35,22)(40,26)
\drawcurve(40,50)(35,54)(40,58)
\drawcurve(40,66)(35,70)(40,74)
\drawcurve(40,42)(30,62)(40,82)
\drawcircle(22.5,35,4)
\end{picture}}
\put(90,48){=}
\put(94,0){\begin{picture}(80,80)
\drawrect[linewidth=.3](5,12,75,88)
\multiput(5,18)(0,8){9}{\line(-1,0){2}}
\multiput(75,18)(0,8){9}{\line(1,0){2}}
\drawline(5,18)(40,50)
\drawline(5,58)(40,82)
\drawline(5,74)(40,74)
\drawcurve(5,26)(15,33)(5,42)
\drawcurve(5,34)(15,42)(5,50)
\drawcurve(5,66)(15,74)(5,82)
\drawcurve(40,18)(35,22)(40,26)
\drawcurve(40,34)(35,38)(40,42)
\drawcurve(40,58)(35,62)(40,66)
\drawcircle(22.5,50,3)
\drawcircle(22.5,50,6)
\drawcircle(22.5,50,9)
\drawline[linewidth=.3,dash={2.05}0](40,12)(40,88)
\put(35,0){\begin{picture}(40,80)
\drawline(5,66)(40,34)
\drawcurve(5,18)(10,22)(5,26)
\drawcurve(5,34)(15,46)(5,58)
\drawcurve(5,42)(10,46)(5,50)
\drawcurve(5,74)(10,78)(5,82)
\drawcurve(40,18)(35,22)(40,26)
\drawcurve(40,50)(35,54)(40,58)
\drawcurve(40,66)(35,70)(40,74)
\drawcurve(40,42)(30,62)(40,82)
\drawcircle(22.5,35,4)
\end{picture}}
\end{picture}}
\end{picture}
\caption{Multiplication of wire diagrams}\label{fig:multiplication}
\end{figure}

In more precise terms, if $\xi=(\pi_1;d_1)$, $\eta=(\pi_2;d_2)$,
then a left pin $p$ and a right pin $q'$ of the product $\xi\eta$
are connected by a $t$-wire if and only if one of the following
conditions holds:
\begin{trivlist}
\item[$\bullet$] \wire{p}{u'} is a $t$-wire in $\xi$ and \wirei{u}{q'} is a
$t$-wire in $\eta$ for some $u\in[n]$;
\item[$\bullet$] for some $s>1$ and some $u_1,v_1,u_2,\dots,v_{s-1},u_s\in[n]$
(all pairwise distinct), \wire{p}{u_1'} is a $t$-wire in $\xi$ and
\wirei{u_s}{q'} is a $t$-wire in $\eta$, while all
\wirei{u_i}{v_i} are $\ell$-wires in $\eta$ and all
\wireii{v_i'}{u_{i+1}'} are $r$-wires in $\xi$.
\end{trivlist}
An analogous characterization holds for the $\ell$-wires and
$r$-wires of the product. Each extra circle of $\xi\eta$
corresponds to a sequence $u_1,v_1,\dots,u_s,v_s\in[n]$ with $s\ge
1$ and pairwise distinct $u_1,v_1,\dots,u_s,v_s$ such that all
\wirei{u_i}{v_i} are $\ell$-wires in $\eta$, while all
\wireii{v_i'}{u_{i+1}'} and \wirei{v_s'}{u_1'} are $r$-wires in
$\xi$.

It easy to see that the above defined multiplication in $\mathcal{W}_n$ is associative and that the chip with 0 circles and the horizontal
$t$-wires \wires{1}{1'},\dots,\wire{n}{n'} is the identity element with respect to the multiplication. Thus, $\mathcal{W}_n$ is a monoid;
$\mathcal{W}_n$ also admits two natural unary operations. The first of them geometrically amounts to the reflection of each chip along its
vertical symmetry axis. To formally introduce this reflection, consider the permutation ${}^*$ on $[n]\cup[n]'$ that swaps primed with
unprimed elements, that is, set
$$k^*:=k',\ (k')^*:=k\text{ for all } k\in [n].$$
Then define $(\pi;d)^*:=(\pi^*;d)$, where
$$\pi^*:=\bigl\{\{x^*,y^*\}\mid \{x,y\} \text{ is a block of } \pi\bigr\}.$$
It is easy to verify that
$$ \xi^{**}=\xi,\ (\xi\eta)^*=\eta^*\xi^*\ \text{ for all }\ \xi,\eta\in\mathcal{W}_n,$$
hence the operation $\xi\mapsto\xi^*$ is an \emph{involution} of $\mathcal{W}_n$. The second unary operation on $\mathcal{W}_n$ rotates
each chip by the angle of $180$ degrees. To define it formally, let
$$\alpha:=\Bigl(\bigr\{\{1,n'\}, \{2,(n-1)'\},\dots,\{n,1'\}\bigr\};\,0\Bigr)$$
and define the unary operation $^\rho:\mathcal{W}_n\to\mathcal{W}_n$ by
$$\xi^{\rho}:= \alpha\xi^*\alpha.$$
Since $\alpha^*=\alpha$ and $\alpha^2=1$, we get that $\xi\mapsto\xi^\rho$ is also an involution on $\mathcal{W}_n$. We refer to the
involutions $^*$ and $^\rho$ as the \emph{reflection} and respectively the \emph{rotation}.

\citet{Ka90} defined the \emph{connection monoid} $\mathcal{C}_n$ as the submonoid of the wire monoid $\mathcal{W}_n$ consisting of all
elements of $\mathcal{W}_n$ that have a representation as a chip whose wires do not cross. He has shown that $\mathcal{C}_n$ is generated
by the \emph{hooks} $h_1,\dots,h_{n-1}$, where
$$h_i:=\Bigl(\bigr\{\{i,i+1\},\{i',(i+1)'\},\{j,j'\}\mid \text{for all } j\ne i,i+1\bigr\};\,0\Bigr),$$
and the \emph{circle} $c:=\Bigl(\bigr\{\{j,j'\}\mid \text{for all } j=1,\dots,n\bigr\};\,1\Bigr),$ see Figure~\ref{fig:hooks} for an
illustration.
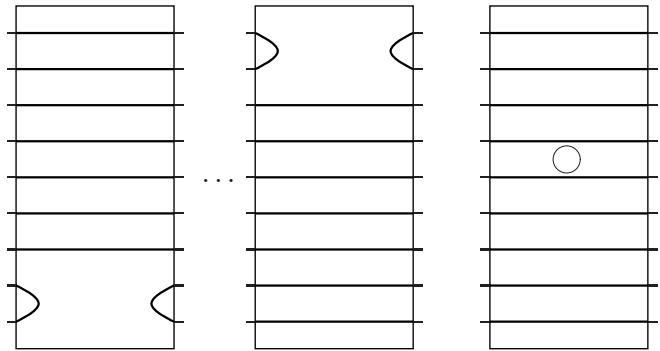
\begin{figure}[b]
\centering \unitlength=.6mm
\begin{picture}(150,80)(0,10)
\gasset{AHnb=0,linewidth=.5}
\drawrect[linewidth=.3](5,12,40,88)
\multiput(5,18)(0,8){9}{\line(-1,0){2}}
\multiput(40,18)(0,8){9}{\line(1,0){2}}
\drawcurve(5,18)(10,22)(5,26)
\drawline(5,34)(40,34)
\drawline(5,42)(40,42)
\drawline(5,50)(40,50)
\drawline(5,58)(40,58)
\drawline(5,66)(40,66)
\drawline(5,74)(40,74)
\drawline(5,82)(40,82)
\drawcurve(40,18)(35,22)(40,26)
\put(46,49){$\dots$}
\put(53,0){\begin{picture}(50,80)
\drawrect[linewidth=.3](5,12,40,88)
\multiput(5,18)(0,8){9}{\line(-1,0){2}}
\multiput(40,18)(0,8){9}{\line(1,0){2}}
\drawline(5,18)(40,18)
\drawline(5,26)(40,26)
\drawline(5,34)(40,34)
\drawline(5,42)(40,42)
\drawline(5,50)(40,50)
\drawline(5,58)(40,58)
\drawline(5,66)(40,66)
\drawcurve(5,74)(10,78)(5,82)
\drawcurve(40,74)(35,78)(40,82)
\end{picture}}
\put(105,0){\begin{picture}(50,80)
\drawrect[linewidth=.3](5,12,40,88)
\multiput(5,18)(0,8){9}{\line(-1,0){2}}
\multiput(40,18)(0,8){9}{\line(1,0){2}}
\drawline(5,18)(40,18)
\drawline(5,26)(40,26)
\drawline(5,34)(40,34)
\drawline(5,42)(40,42)
\drawline(5,50)(40,50)
\drawline(5,58)(40,58)
\drawline(5,66)(40,66)
\drawline(5,74)(40,74)
\drawline(40,82)(5,82)
\put(22,54){\circle{6}}
\end{picture}}
\end{picture}
\caption{The hooks $h_1,\dots,h_8$ and the circle $c$ in $\mathcal{C}_9$}\label{fig:hooks}
\end{figure}
It is immediate to check that the generators $h_1,\dots,h_{n-1},c$ satisfy the relations \eqref{eq:TL1}, \eqref{eq:TL2}, and
\eqref{eq:TL4}, whence there exists a homomorphism from the Kauffman monoid $\mathcal{K}_n$ onto the connection monoid $\mathcal{C}_n$. In
fact, this homomorphism turns out to be an isomorphism between $\mathcal{K}_n$ and $\mathcal{C}_n$; a proof was outlined in \citep{Ka90}
and presented in full detail in \citep{BDP02}.

Observe that the set $\{h_1,\dots,h_{n-1},c\}$ is closed under both the reflection and the rotation in $\mathcal{W}_n$: the reflection
fixes each generator, while the rotation fixes $c$ and maps $h_i$ to $h_{n-i}$ for each $i=1,\dots,n-1$. Therefore, the submonoid
$\mathcal{C}_n$ generated by $\{h_1,\dots,h_{n-1},c\}$ is also closed under these involutions that, of course, transfer to the isomorphic
monoid $\mathcal{K}_n$, as well. The reader who prefers to have a `picture-free' definition of the two involutions in Kauffman monoids may
observe that the relations \eqref{eq:TL1}, \eqref{eq:TL2}, and \eqref{eq:TL4} are left-right symmetric: each of these relations coincides
with its mirror image. Therefore, the map that fixes each generator of the monoid $\mathcal{K}_n$ uniquely extends to an involution of
$\mathcal{K}_n$; clearly, this extension is nothing but the reflection ${}^*$, and this gives a purely syntactic definition of the latter.
In a similar way, one can give a syntactic definition of the rotation ${}^\rho$: it is a unique involutary extension of the map that fixes
$c$ and swaps $h_i$ and $h_{n-i}$ for each $i=1,\dots,n-1$.

Since the involutions $\xi\mapsto\xi^*$ and $\xi\mapsto\xi^\rho$ (especially the first one) are essential for many applications of Kauffman
monoids, we find it appropriate to extend our study of the finite basis problem for the identities holding in $\mathcal{K}_n$ also to their
identities as algebras of type (2,1), with the reflection or the rotation in the role of the unary operation. The corresponding question
was stated in the last-named author's lecture mentioned in the introduction; here we will give a complete answer to it.

Let us return for a moment to the wire monoid $\mathcal{W}_n$. Denote by $\mathcal{B}_n$ the set of all $2n$-pin chips without circles, in
other words, the set of all partitions of $[n]\cup [n]'$ into 2-element blocks. Observe that this set is finite. We define the
multiplication of two chips in $\mathcal{B}_n$ as follows: we multiply the chips as elements of $\mathcal{W}_n$ and then reduce the product
to a chip in $\mathcal{B}_n$ by removing all circles. This multiplication makes $\mathcal{B}_n$ a monoid known as the \emph{Brauer monoid}:
the monoids $\mathcal{B}_n$ were introduced by \citet{Br37} as vector space bases of certain associative algebras relevant in
representation theory and thus became the historically first species of diagram monoids. We stress that even though the base set of
$\mathcal{B}_n$ has been defined as a subset in the base set of $\mathcal{W}_n$, it is \textbf{not} true that $\mathcal{B}_n$ forms a
submonoid of $\mathcal{W}_n$. On the other hand, it is easy to see that the `forgetting' map $\varphi:\mathcal{W}_n\to\mathcal{B}_n$
defined by $\varphi(\pi;d)=\pi$ is a surjective homomorphism (the homomorphism just forgets the circles of its argument).

Clearly, both the reflection and the rotation respect
$\mathcal{B}_n$ as a set and behave as anti-isomorphisms with
respect to multiplication in $\mathcal{B}_n$. Thus,
$\mathcal{B}_n$ forms an involution monoid under each of these
unary operations; moreover, the homomorphism $\varphi$ is
compatible with both involutions $^*$ and $^\rho$.
We summarize and augment the above information about the wire
monoids and the Brauer monoids in the following lemma.

\begin{lemma}
\label{lm:brauer} For each $n\ge2$, the map $\varphi:(\pi;d)\mapsto\pi$ is a homomorphism from the monoid $\mathcal{W}_n$ onto the finite
monoid $\mathcal{B}_n$; the homomorphism respects both involutions $^*$ and $^\rho$. For every idempotent in $\mathcal{B}_n$, its inverse
image under $\varphi$ is a commutative subsemigroup in $\mathcal{W}_n$.
\end{lemma}

\begin{proof}
It remains to verify the last claim of the lemma. By the definition of $\varphi$, for each $\pi\in\mathcal{B}_n$, its inverse image under
$\varphi$ coincides with the set
$$\Pi:=\{(\pi;d)\mid d=0,1,\dots\}.$$
If $\pi^2=\pi$ in the Brauer monoid, then the product $(\pi;0)(\pi;0)$ in the wire monoid belongs to $\Pi$ whence $(\pi;0)(\pi;0)=(\pi;m)$
for some nonnegative integer $m$. Now if we multiply two arbitrary elements $(\pi;k),(\pi;\ell)\in\Pi$, we get $(\pi;k+\ell+m)$
independently of the order of the factors.
\end{proof}

The \emph{Jones monoid}\footnote{The name was suggested by
\citet{LF06} to honor the contribution of V.F.R.~Jones to the theory
\citep[see, e.g.,][Section~4]{Jo83}.} $\mathcal{J}_n$ can be
defined as the submonoid of the Brauer monoid $\mathcal{B}_n$
consisting of all elements of $\mathcal{B}_n$ that have a
representation as a chip whose wires do not cross. Thus,
$\mathcal{J}_n$ relates to $\mathcal{B}_n$ precisely as the
Kauffman monoid $\mathcal{K}_n$ (in its incarnation as the
connection monoid $\mathcal{C}_n$) relates to the wire monoid
$\mathcal{W}_n$. Alternatively, one can define the Jones monoid as
the image of the Kauffman monoid under the restriction of the
`forgetting' homomorphism $\varphi$ to the latter. Clearly,
$\mathcal{J}_n$ is closed under $^*$ and $^\rho$ and forms an
involution monoid with respect to each of these operations. The
following scheme summarizes the relations between the four species
of diagram monoids introduced so far:
$$
\begin{CD}
\mathcal{W}_n @>\varphi>> \mathcal{B}_n\\
@AAA                       @AAA \\
\mathcal{K}_n @>\varphi>> \mathcal{J}_n
\end{CD}.
$$
The vertical arrows here stand for embeddings, the horizontal ones for surjections, and all maps respect multiplication and both
involutions.

The following fact is just a specialization of Lemma~\ref{lm:brauer}.

\begin{lemma}
\label{lm:jones} For each $n\ge2$, the map $\varphi:(\pi;d)\mapsto\pi$ is a homomorphism from the monoid $\mathcal{K}_n$ onto the finite
monoid $\mathcal{J}_n$; the homomorphism respects both involutions $^*$ and $^\rho$. For every idempotent in $\mathcal{J}_n$, its inverse
image under $\varphi$ is a commutative subsemigroup in $\mathcal{K}_n$.
\end{lemma}

As promised at the beginning of this section, we conclude with showing how one may bypass geometric considerations and define the Jones
monoid in terms of generators and relations. Since the monoid $\mathcal{J}_n$ is the image of $\mathcal{K}_n$ under $\varphi$, it is
generated by the hooks $h_1,\dots,h_{n-1}$ and the following relations hold in $\mathcal{J}_n$:
$$h_{i}h_{j}=h_{j}h_{i}\qquad\qquad \text{if } |i-j|\ge 2,\ i,j=1,\dots,n-1;$$
\begin{align}
&h_{i}h_{j}h_{i}=h_{i} &&\text{if } |i-j|=1,\ i,j=1,\dots,n-1;\label{eq:jones}\\
&h_{i}h_{i}=h_{i}.     &&\text{for each } i=1,\dots,n-1.\notag
\end{align}
In fact, it can be verified \citep{BDP02} that the monoid generated by $h_1,\dots,h_{n-1}$ subject to the relations \eqref{eq:jones}, i.e.,
the monoid that spans the Temperley--Lieb algebra $\mathcal{TL}_n(\delta)$ with $\delta=1$, is isomorphic to $\mathcal{J}_n$. Thus, one can
define $\mathcal{J}_n$ by this presentation. Lemma~\ref{lm:jones} can be then recovered as follows. The homomorphism
$\varphi:\mathcal{K}_n\twoheadrightarrow\mathcal{J}_n$ arises in this setting as a unique homomorphic extension of the map that sends the
generators $h_1,\dots,h_{n-1}$ of $\mathcal{K}_n$ to the generators of $\mathcal{J}_n$ with the same names and `erases' the generator $c$
by sending it to 1; the fact that such an extension exists and enjoys all properties registered in Lemma~\ref{lm:jones} readily follows
from the close similarity between the relations \eqref{eq:TL1}, \eqref{eq:TL2}, \eqref{eq:TL4} on the one hand and the relations
\eqref{eq:jones} on the other hand. The only claim in Lemma~\ref{lm:jones} which is not that apparent with this definition of
$\mathcal{J}_n$ is the finiteness of the monoid. This indeed requires some work \citep[see][for details]{BDP02}. From the diagrammatic
representation it can be easily calculated that the cardinality of $\mathcal{J}_n$ is the $n$-th Catalan number
$\dfrac{1}{n+1}\dbinom{2n}{n}$. For further interesting results concerning the monoids $\mathcal{K}_n$, $\mathcal{J}_n$ and similarly
defined ones the reader may consult \citep{DP03}.

\section{A sufficient condition for the non-existence of a finite basis}
\label{sec:nfb}

We assume the reader's familiarity with basic concepts of the theory of varieties \citep[see, e.g.,][Chapter~II]{BS81} and of semigroup
theory \citep[see, e.g.,][Chapter~1]{CP61}.

We aim to establish a condition for the nonfinite basis property that would apply to both `plain' semigroups and semigroups with involution
as algebras of type (2,1). The two cases have much in common, and we use square brackets to indicate adjustments to be made in the
involution case. First, let us formally introduce \iss.

An algebra $\mathcal{S}=\langle S,\cdot,{}^\star\rangle$ of type (2,1) is called an \emph{\is} if $\langle S,\cdot\rangle$ is a semigroup
(referred to as the \emph{semigroup reduct} of $\mathcal{S}$) and the identities
$$(xy)^\star\bumpeq y^\star x^\star \text{ and  } (x^\star)^\star\bumpeq x$$
hold, in other words, if the unary operation $x\mapsto x^\star$ is an involutory anti-automorphism of $\langle S,\cdot\rangle$.

The \emph{free \is} $\FI(X)$ on a given alphabet $X$ can be constructed as follows.  Let $\overline{X}:=\{x^\star\mid x\in X\}$ be a
disjoint copy of $X$. Define $(x^\star)^\star:=x$ for all $x^\star\in \overline{X}$. Then $\FI(X)$ is the free \sgp\
$(X\cup\overline{X})^+$ endowed with the involution defined by
$$(x_1\cdots x_m)^\star:= x_m^\star\cdots x_1^\star$$
for all $x_1,\dots,x_m\in X\cup \overline{X}$. We refer to elements of $\FI(X)$ as \emph{involutory words over $X$} while elements of the
free semigroup $X^+$ will be referred to as \emph{plain words over $X$}.

If an \is\ $\mathcal{T}=\langle T,\cdot,{}^\star\rangle$ is generated by a set $Y\subseteq T$, then every element in $\mathcal{T}$ can be
represented by an involutory word over $Y$ and thus by a plain word over $Y\cup\overline{Y}$ where $\overline{Y}=\{y^\star\mid y\in Y\}$.
Hence the reduct $\langle T,\cdot\rangle$ is generated by the set $Y\cup\overline{Y}$; in particular, $\mathcal{T}$ is \fg\ if and only if
so is $\langle T,\cdot\rangle$. Recall that an algebra is said to be \emph{locally finite} if each of its \fg\ subalgebras is finite. From
the above remark, it follows that an \is\ $\mathcal{S}=\langle S,\cdot,{}^\star\rangle$ is locally finite if and only if so is $\langle
S,\cdot\rangle$. We denote by $\mathbf{L}$ the class of all locally finite semigroups. A variety of [involution] semigroups is
\emph{locally finite} if all its members are locally finite. Given a class $\mathbf{K}$ of [involution] semigroups, we denote by
$\var\mathbf{K}$ the variety of [involution] \sgps\ it generates; if $\mathbf{K}=\{\mathcal{S}\}$, we write $\var\mathcal{S}$ rather than
$\var\{\mathcal{S}\}$.

Let $\mathbf{A}$ and $\mathbf{B}$ be two subclasses of a fixed class $\mathbf{C}$ of algebras. The \emph{Mal'cev product}
$\mathbf{A}\malcev\mathbf{B}$ of $\mathbf{A}$ and $\mathbf{B}$ (within $\mathbf{C}$) is the class of all algebras
$\mathcal{C}\in\mathbf{C}$ for which there exists a congruence $\theta$ such that the quotient algebra $\mathcal{C}/\theta$ lies in
$\mathbf{B}$ while all $\theta$-classes that are subalgebras in $\mathcal{C}$ belong to $\mathbf{A}$. Note that for a congruence $\theta$
on a semigroup $\mathcal{S}$, a congruence class $s\theta$ forms a subsemigroup of $\mathcal{S}$ if and only if the element $s\theta$ is an
idempotent of the quotient $\mathcal{S}/\theta$. Of essential use will be a powerful result by \citet{Br68,Br71} that can be stated in
terms of the Mal'cev product as follows.

\begin{proposition}[\mdseries\citep{Br68,Br71}]
\label{pr:brown} $\mathbf{L}\malcev\mathbf{L}=\mathbf{L}$ where the
Mal'cev product is considered within the class of all semigroups.
\end{proposition}

Let $x_1,x_2,\dots,x_n,\dots$ be a sequence of letters. The sequence $\{Z_n\}_{n=1,2,\dots}$ of \emph{Zimin words} is defined inductively
by $Z_1:=x_1$, $Z_{n+1}:=Z_nx_{n+1}Z_n$. We say that a word $v$ is an [\emph{involutory}] \emph{isoterm} for a class $\mathbf{C}$ of
semigroups [with involution] if the only [involutory] word $v'$ such that all members of $\mathbf{C}$ satisfy the [involution] semigroup
identity $v\bumpeq v'$ is the word $v$ itself.

If a semigroup $\mathcal{S}$ satisfies the identities $x^2y\bumpeq
x^2\bumpeq yx^2$, then $\mathcal{S}$ has a zero and the value of
the word $x^2$ in $\mathcal{S}$ under every evaluation of the
letter $x$ in $\mathcal{S}$ is equal to zero. Having this in mind,
we use the expression $x^2\bumpeq 0$ as an abbreviation for the
identities $x^2y\bumpeq x^2\bumpeq yx^2$.

The last ingredient that we need comes from
\citep[Proposition~3]{Sa87} for the plain case and from
\citep[Corollary~2.6]{ADV12a} for the involution case.

\begin{proposition}[\mdseries\citep{Sa87,ADV12a}]
\label{pr:sapir} Let $\mathbf{V}$ be a variety of $[$involution$]$ \sgps. If
\begin{enumerate}
\item[(i)] all members of\, $\mathbf{V}$ satisfying $x^2\bumpeq 0$ are
locally finite and
\item [(ii)] each Zimin word is an $[$involutory$]$ isoterm relative to $\mathbf{V}$,
\end{enumerate}
then $\mathbf{V}$ is \nfb.
\end{proposition}

In the following we shall present a specialization of Proposition \ref{pr:sapir} by presenting a sufficient condition for a variety
$\mathbf{V}$ to satisfy condition (i). An essential step towards this result is the next lemma whose proof is a refinement of one of the
crucial arguments in \citep{SV94}. Here $\mathbf{Com}$ denotes the variety of all commutative semigroups.

\begin{lemma}
\label{lm:core} Let $\mathcal{T}$ be a semigroup in $\mathbf{Com}\malcev\mathbf{L}$ and let $I$ be the ideal of $\mathcal{T}$ generated by
$\{t^2\mid t\in\mathcal{T}\}$. Then the Rees quotient $\mathcal{T}/I$ is locally finite.
\end{lemma}

\begin{proof}
Let $\alpha$ be a congruence on $\mathcal{T}$ such that $\mathcal{T}/\alpha$ is locally finite and the idempotent $\alpha$-classes are
commutative subsemigroups of $\mathcal{T}$. Let $\rho_I$ be the Rees congruence of $\mathcal{T}$ corresponding to the ideal $I$ and
$\beta=\alpha\cap\rho_I$. We have the following commutative diagram in which all homomorphisms are canonical projections.
\begin{center}
\unitlength=.5mm
\begin{picture}(60,35)
\gasset{Nframe=n,AHnb=2,AHlength=0}
\node(N)(30,35){$\mathcal{T}$}
\node(S)(30,5){}
\node(Y)(30,3){$\mathcal{T}/\beta$}
\node(W)(0,20){}
\node(Z)(-2,20){$\mathcal{T}/\alpha$}
\node(E)(60,20){}
\node(X)(75,20){$\mathcal{T}/\rho_I=\mathcal{T}/I$}
\drawedge(N,S){}
\drawedge(N,W){}
\drawedge(N,E){}
\drawedge(S,W){}
\drawedge(S,E){}
\end{picture}
\end{center}

Recall that a semigroup is said to be \emph{periodic} if each of its one-generated subsemigroups is finite. The semigroup
$\mathcal{T}/\alpha$ is locally finite and thus periodic. Moreover, since the restrictions of $\alpha$ and $\beta$ to the ideal $I$
coincide, we have $I/\alpha=I/\beta$ whence $I/\beta$ is periodic, as well. Since for each element of $\mathcal{T}/\beta$, its square
belongs to $I/\beta$, it follows that $\mathcal{T}/\beta$ is also periodic, and so is each subsemigroup of $\mathcal{T}/\beta$.

Now let $A\in\mathcal{T}/\alpha$ be an idempotent $\alpha$-class; by assumption, $A$ is a commutative subsemigroup of $\mathcal{T}$. Then
the inverse image of $A$ (considered as an element of $\mathcal{T}/\alpha$) under the canonical projection
$\mathcal{T}/\beta\twoheadrightarrow\mathcal{T}/\alpha$ is the subsemigroup $A/\beta$ of $\mathcal{T}/\beta$, and this subsemigroup is at
the same time commutative and periodic. It is well known (and easy to verify) that every commutative periodic semigroup is locally finite.
We see that the congruence $\alpha/\beta$ on $\mathcal{T}/\beta$ satisfies the two conditions: (a) the quotient
$(\mathcal{T}/\beta)\text{\Large/}(\alpha/\beta)\cong\mathcal{T}/\alpha$ is locally finite and (b) the $\alpha/\beta$-classes which are
subsemigroups are locally finite. By Proposition~\ref{pr:brown}, $\mathcal{T}/\beta$ is itself  locally finite, and so is its quotient
$\mathcal{T}/I$.
\end{proof}

For two  \sgp\ varieties $\mathbf{V}$ and $\mathbf{W}$, their Mal'cev product $\mathbf{V}\malcev\mathbf{W}$ within the class of all
semigroups may fail to be a variety but it is always closed under forming sub\sgps\ and direct products \citep[see][Theorems~1
and~2]{Ma67}. Therefore the variety $\var(\mathbf{V}\malcev\mathbf{W})$ generated by $\mathbf{V}\malcev\mathbf{W}$ is comprised of all
homomorphic images of the members of $\mathbf{V}\malcev\mathbf{W}$. We are now in a position to formulate and to prove our main result.

\begin{theorem} \label{thm:main} A variety $\mathbf{V}$ of $[$involution$]$ semigroups is nonfinitely based provided that
\begin{enumerate}
\item[(i)] $[$the class of all semigroup reducts of\,$]$ $\mathbf{V}$ is contained in the variety
$\var(\mathbf{Com}\malcev\mathbf{W})$ for some locally finite semigroup variety $\mathbf{W}$ and
\item[(ii)] each Zimin word is an $[$involutory$]$ isoterm relative to $\mathbf{V}$.
\end{enumerate}
\end{theorem}

\begin{proof}
By Proposition~\ref{pr:sapir}, it suffices to verify that all members of $\mathbf{V}$ satisfying $x^2\bumpeq 0$ are locally finite. Since
an involution semigroup is locally finite if and only if so is its semigroup reduct, it suffices to do so for the semigroup reducts of the
members of $\mathbf{V}$. Let $\mathbf{W}$ be a locally finite semigroup variety as per condition (i). We need to check that each semigroup
$\mathcal{S}\in\var(\mathbf{Com}\malcev\mathbf{W})$ which satisfies $x^2\bumpeq 0$ is locally finite. As we observed prior to the
formulation of the theorem, $\mathcal{S}$ is a homomorphic image of a semigroup $\mathcal{T}\in\mathbf{Com}\malcev\mathbf{W}$; let
$\varphi$ stand for the corresponding homomorphism. Consider the ideal $I$ in $\mathcal{T}$ generated by $\{t^2\mid t\in \mathcal{T}\}$.
Then $I\subseteq 0\varphi^{-1}$, and therefore, the homomorphism $\varphi$ factors through $\mathcal{T}/I$ which is locally finite by Lemma
\ref{lm:core}. Consequently, $\mathcal{S}$ is also locally finite.
\end{proof}

\begin{remark}
It follows immediately from the proof of Lemma \ref{lm:core} that Theorem~\ref{thm:main} remains valid if we replace the variety
$\mathbf{Com}$ of all commutative semigroups by an arbitrary semigroup variety all of whose periodic members are locally finite.
\end{remark}

\begin{remark}
For a locally finite [involution] semigroup variety $\mathbf{V}$, condition (i) is trivially satisfied with $\mathbf{W}=\mathbf{V}$. In
this case, condition (ii) is sufficient for $\mathbf{V}$ to be \nfb; moreover, $\mathbf{V}$ then is even \emph{inherently nonfinitely
based}, i.e., it is not contained in any \fb\ locally finite variety. The corresponding result is captured by \citep{Sa87} for plain \sgps\
and by \citep{ADV12a} for \iss. It follows that the novelty in the present paper, though not always explicitly mentioned, is about
\textbf{infinite} [involution] semigroups, or, to be more precise, [involution] semigroups which do not generate a locally finite variety.
\end{remark}

\begin{remark}
Proposition \ref{pr:sapir} and therefore Theorem \ref{thm:main} formulate, in fact, sufficient conditions that the variety in question be
not only \nfb\ but even be \textbf{of infinite axiomatic rank}, that is, there is no basis for the equational theory that uses only
finitely many variables. Consequently, in all our applications, the respective [involution] semigroups are also not only \nfb\ but even of
infinite axiomatic rank. This is worth registering because an infinite [involution] semigroup can be \nfb\ but of finite axiomatic rank.
\end{remark}

\begin{remark}
\label{rm:interval} If two given varieties $\mathbf{X}$ and $\mathbf{Y}$ of [involution] semigroups satisfy
$\mathbf{X}\subseteq\mathbf{Y}$, and $\mathbf{Y}$ satisfies condition (i) while $\mathbf{X}$ satisfies condition (ii), then all varieties
$\mathbf{V}$ such that $\mathbf{X}\subseteq \mathbf{V}\subseteq\mathbf{Y}$ satisfy both conditions, and therefore, are \nfb. Stated this
way, Theorem~\ref{thm:main} may be used to produce intervals consisting entirely of \nfb\ varieties in the lattice of [involution]
semigroup varieties. We conclude this section with an example of such an application.

For two varieties $\mathbf{V}$ and $\mathbf{W}$, we denote by $\mathbf{V}\vee\mathbf{W}$ their \emph{join}, i.e., the least variety
containing both $\mathbf{V}$ and $\mathbf{W}$. \citet{SV94} proved that for each locally finite \sgp\ variety $\mathbf{W}$ which contains
the variety $\mathbf{B}$ of all \emph{bands} (idempotent \sgps), the join $\mathbf{Com}\vee\mathbf{W}$ is \nfb. More precisely, in
\citet{SV94} it is shown that each Zimin word is an isoterm relative to $\mathbf{Com}\vee\mathbf{B}$ and each member of $\mathbf{Com}\vee
\mathbf{W}$ which satisfies the identity $x^2\bumpeq 0$ is locally finite (the latter by an argument that has been refined in the proof of
Lemma \ref{lm:core}). By Theorem \ref{thm:main} it follows that each variety $\mathbf{V}$ for which $\mathbf{Com}\vee\mathbf{B}\subseteq
\mathbf{V}\subseteq\var(\mathbf{Com}\malcev\mathbf{W})$ is \nfb. Notice that $\mathbf{Com}\vee\mathbf{W}\subseteq
\var(\mathbf{Com}\malcev\mathbf{W})$ so that the quoted result from \citep{SV94} appears as  special case.

One can obtain an analogous result for involution semigroups if $\mathbf{B}$ is replaced by the variety $\mathbf{B}^\star$ of all bands
with involution and commutative semigroups are considered to be equipped with trivial involution (for the verification that Zimin words are
involutory isoterms relative to $\mathbf{Com}\vee\mathbf{B}^\star$ one can use Lemma \ref{lm:twisted} formulated in the next section).
\end{remark}

\section{Applications}
\label{sec:applications}

For every $n$ there is an injective semigroup homomorphism $\mathcal{K}_n\hookrightarrow \mathcal{K}_{n+1}$ (induced by the map $c\mapsto
c$, $h_i\mapsto h_i$ for $i=1,\dots,n-1$) which is compatible with the reflection. Consequently, for every $n$ we have the inclusion
$\var\mathcal{K}_n\subseteq \var\mathcal{K}_{n+1}$. As mentioned earlier, $\mathcal{K}_n\le\mathcal{W}_n$ whence
$\var\mathcal{K}_n\subseteq \var\mathcal{W}_n$ for every $n$. These inclusions are true if the respective structures are considered either
as semigroups or as involution semigroups with respect to the reflection. We start with applying Theorem~\ref{thm:main} to the Kauffman
monoids $\mathcal{K}_n$ and the wire monoids $\mathcal{W}_n$ with $n\ge3$.

\begin{theorem}
\label{thm:kauffman} Let $n\ge3$ and consider $\mathcal{K}_3$ and $\mathcal{W}_n$ either as semigroups or as involution semigroups with
respect to reflection. Then every $[$involution$]$ semigroup variety $\mathbf{V}$ such that
$\var\mathcal{K}_3\subseteq\mathbf{V}\subseteq\var\mathcal{W}_n$ is \nfb.
\end{theorem}

\begin{proof}
We invoke Theorem~\ref{thm:main} in the form of Remark~\ref{rm:interval} and show that $\var\mathcal{W}_n$ satisfies (i) and
$\var\mathcal{K}_3$ satisfies (ii). Thus, we are to check that the semigroup $\mathcal{W}_n$ belongs to the Mal'cev product of
$\mathbf{Com}$ with a~locally finite semigroup variety and that each Zimin word is an [involutory] isoterm relative to $\mathcal{K}_3$.

The first claim readily follows from Lemma~\ref{lm:brauer}. Indeed, by this lemma there is a homomorphism
$\varphi:\mathcal{W}_n\twoheadrightarrow \mathcal{B}_n$ with the property that for every idempotent in $\mathcal{B}_n$, its inverse image
under $\varphi$ is a commutative subsemigroup in $\mathcal{W}_n$. This immediately yields that $\mathcal{W}_n$ belongs to the Mal'cev
product $\mathbf{Com}\malcev\var\mathcal{B}_n$, and $\var\mathcal{B}_n$ is locally finite as the variety generated by a finite algebra
\cite[see][Theorem~10.16]{BS81}.

In order to show that Zimin words are isoterms relative to $\mathcal{K}_3$, consider the ideal $C$ of $\mathcal{K}_3$ generated by $c$.
Clearly, $\mathcal{K}_3\setminus C=\{1,h_1,h_2,h_1h_2,h_2h_1\}$. If we denote the images of $h_1$ and $h_2$ in the Rees quotient
$\mathcal{K}_3/C$ by $a$ and $b$ respectively, then the relations of $\mathcal{K}_3$ translate into the following relations for $a$ and
$b$:
$$a^2=0,\ b^2=0,\ aba=a,\ bab=b.$$
These relations define the 6-element Brandt monoid $B_2^1$ (in the class of all monoids with 0). Thus, $\mathcal{K}_3/C$ satisfies the
relations of $B_2^1$ and the Rees quotient also consists of 6 elements, so that $\mathcal{K}_3/C\cong B_2^1$. It is well known
\citep[see][Lemma 3.7]{Sa87} that each Zimin word is an isoterm relative to $B_2^1$. This completes the proof in the plain \sgp\ case.

If we consider $\mathcal{K}_3$ as an \is\ under reflection, we can employ the approach of \citet{ADPV14}. Recall that the 3-element
\emph{twisted semilattice} is the \is\ $\TSL=\langle\{e,f,0\},\cdot,{}^\star\rangle$ in which $e^2=e$, $f^2=f$ and all other products are
equal to $0$, while the unary operation is defined by $e^\star=f$, $f^\star=e$, and $0^\star=0$. The following observation has been made in
the proof of Theorem~3.1 in \citep{ADPV14}.

\begin{lemma}
\label{lm:twisted} Let $\mathcal{T}=\langle T,\cdot,{}^\star\rangle$ be an \is\ such that each Zimin word is an isoterm relative to its
\sgp\ reduct $\langle T,\cdot\rangle$. If the $3$-element twisted semilattice $\TSL$ belongs to the variety $\var\mathcal{T}$, then each
Zimin word is also an involution isoterm relative to $\mathcal{T}$.
\end{lemma}

Clearly, the ideal $C$ of $\mathcal{K}_3$ is closed under reflection, which therefore induces an involution on $\mathcal{K}_3/C\cong B_2^1$.
The latter involution swaps the idempotents $ab$ and $ba$ and fixes all other elements of $B_2^1$ whence the subset $\{ab,ba,0\}$ of
$B_2^1$ constitutes an involution subsemigroup isomorphic to $\TSL$. Hence $\TSL$ belongs to the variety generated by $\mathcal{K}_3$ as an
\is\ under reflection and Lemma~\ref{lm:twisted} applies.
\end{proof}

The situation is somewhat more delicate if we consider $\mathcal{K}_n$ and $\mathcal{W}_n$ as \iss\ under rotation; we denote these \iss\
by $\mathcal{K}_n^\rho$ and $\mathcal{W}_n^\rho$ respectively. For every $n$ we have the following embeddings.
\begin{trivlist}
\item[$\bullet$] $\mathcal{K}_n^\rho\hookrightarrow\mathcal{K}_{n+2}^\rho$ and $\mathcal{W}_n^\rho\hookrightarrow\mathcal{W}_{n+2}^\rho$.
These embeddings are obtained by adding one $t$-wire on top and one on bottom of each chip; for the case of Kauffman monoids, the embedding
can be alternatively defined in terms of generators: it is induced by the map $c\mapsto c$, $h_i\mapsto h_{i+1}$ for $i=1,\dots,n-1$.
\item[$\bullet$]  $\mathcal{K}_n^\rho\hookrightarrow\mathcal{K}_{2n}^\rho$ and $\mathcal{W}_n^\rho\hookrightarrow\mathcal{W}_{2n}^\rho$.
These embeddings are obtained by `doubling' each chip; in terms of generators for $\mathcal{K}_n^\rho$, the embedding is  induced by the
map $c\mapsto c^2$, $h_i\mapsto h_ih_{n+i}$ for $i=1,\dots, n-1$.
\item[$\bullet$] $\mathcal{W}_{2n}^\rho\hookrightarrow\mathcal{W}_{2n+1}^\rho$. The embedding is obtained by inserting a $t$-wire just into
the middle of each chip.
\item[$\bullet$]  $\mathcal{K}_n^\rho\hookrightarrow\mathcal{W}_n^\rho$. This is the canonical embedding.
\end{trivlist}
It follows that $\var\mathcal{K}_3^\rho\subseteq\var\mathcal{W}_n^\rho$ for $n=3$ and each $n\ge 5$, and
$\var\mathcal{K}_4^\rho\subseteq\var\mathcal{W}_n^\rho$ for each $n\ge 4$. We do not know whether
$\var\mathcal{K}_3^\rho\subseteq\var\mathcal{W}_4^\rho$ or $\var\mathcal{K}_3^\rho\subseteq\var\mathcal{K}_4^\rho$. In any case, we have a
version of Theorem \ref{thm:kauffman} that is well sufficient for our purposes.

\begin{theorem}
\label{thm:kauffman-rot} Let $m\ge 4$; each variety $\mathbf{V}$ of involution semigroups satisfying
$\var\mathcal{K}_3^\rho\subseteq \mathbf{V}\subseteq \var\mathcal{W}_{m+1}^\rho$ or
$\var\mathcal{K}_4^\rho\subseteq \mathbf{V}\subseteq \mathcal{W}_m^\rho$ is \nfb.
\end{theorem}

\begin{proof}
We have already shown in the proof of Theorem \ref{thm:kauffman} that the semigroup reducts of all members of $\var\mathcal{W}_m^\rho$
satisfying $x^2\bumpeq 0$ are locally finite. In order to apply Theorem~\ref{thm:main} (in the form of Remark~\ref{rm:interval}), it
remains to show that each Zimin word is an involutory isoterm relative to $\var\mathcal{K}_\ell^\rho$ for $\ell=3$ and $\ell=4$. For
$\ell=4$ this follows from the analogous fact for the Jones monoid $\mathcal{J}_4$ considered as an \is\ under rotation (this fact has been
shown in \citep[][Theorem~2.13]{ADV12b}); by Lemma~\ref{lm:jones} the latter monoid is a quotient of $\mathcal{K}_4^\rho$.

It remains to consider the case $\ell=3$. We do not know whether or not $\TSL$ belongs to the variety $\var\mathcal{K}_3^\rho$ hence we do
not know if we can proceed as in the proof of Theorem \ref{thm:kauffman}. Nevertheless, we will show that each Zimin word is an involution
isoterm relative to $\mathcal{K}_3^\rho$.

Arguing by contradiction, assume that for some $n$ and some involutory word $w$, the identity $Z_n\bumpeq w$ holds in $\mathcal{K}_3^\rho$.
First we observe that each letter $x_i$, $i=1,2,\dots,n$, occurs the same number of times in $Z_n$ and $w$. For this, we substitute $c$ for
$x_i$ and 1 for all other letters. The value of the word $Z_n$ under this substitution is $c^{2^{n-i}}$ since it is easy to see that $x_i$
occurs $2^{n-i}$ times in $Z_n$. Similarly, since $c^\rho=c$, the value of $w$ is $c^k$, where $k$ is the number of occurrences of $x_i$ in
$w$. As $Z_n\bumpeq w$ holds in $\mathcal{K}_3^\rho$, the two values should coincide whence $k=2^{n-i}$. In a similar manner one can verify
that the only letters occurring in $w$ are $x_1,x_2,\dots,x_n$.

We have already shown that $Z_n$ is an isoterm relative to $\mathcal{K}_3$ considered as a plain \sgp. Hence $w$ must be a proper
involutory word, that is, it has at least one occurrence of a `starred' letter. We fix an $i\in\{1,2,\dots,n\}$ such that $x_i^\star$
occurs in $w$ and substitute $h_1$ for $x_i$ and 1 for all other letters. It is easy to calculate that the value of the word $Z_n$ under
this substitution is $c^{2^{n-i}-1}h_1$. Since $h_1^\rho=h_2$ in $\mathcal{K}_3^\rho$ and $x_i$ occurs $2^{n-i}$ times in $w$, the word $w$
evaluates to a product $p$ of $2^{n-i}$ factors each of which is either $h_1$ or $h_2$ and at least one of which is $h_2$. As $Z_n\bumpeq
w$ holds in $\mathcal{K}_3^\rho$, the value of $p$ must coincide with $c^{2^{n-i}-1}h_1$, which is only possible when the first and the
last factors of $p$ are $h_1$. Then the relations \eqref{eq:TL2} and \eqref{eq:TL4} ensure that the value of $p$ is $c^kh_1$, where $k$ is
the total number of occurrences of the factors $h_1h_1$ and $h_2h_2$ in $p$. However, $p$ has at least one occurrence of $h_1h_2$ and at
least one occurrence of $h_2h_1$, and therefore $k\le 2^{n-i}-3$, a contradiction.
\end{proof}

\begin{remark}
To get a version of Theorem~\ref{thm:kauffman} that could be stated and justified without any appealing to geometric considerations, one
should change $\mathcal{W}_n$ to $\mathcal{K}_n$ in the formulation of Theorem~\ref{thm:kauffman} and refer to Lemma~\ref{lm:jones} instead
of Lemma~\ref{lm:brauer} in its proof. (Recall that we outlined a `picture-free' proof of Lemma~\ref{lm:jones} at the end of
Section~\ref{sec:wire}.) This reduced version of Theorem~\ref{thm:kauffman} still suffices to solve the finite basis problem for the
identities holding in the Kauffman monoids. The same observation applies to Theorem~\ref{thm:kauffman-rot}.
\end{remark}

\begin{remark}
Theorems~\ref{thm:kauffman} and \ref{thm:kauffman-rot} imply that each of the monoids $\mathcal{W}_n$ and $\mathcal{K}_n$ with $n\ge3$ is
\nfb\ as both a plain \sgp\ and an \is\ with either reflection or rotation. For the sake of completeness, we mention that the monoids
$\mathcal{W}_2$ and $\mathcal{K}_2$ are easily seen to be commutative and hence they are \fb\ by a classic result of \citet{Pe69}.
Moreover, both reflection and rotation act trivially in $\mathcal{W}_2$, and therefore, $\mathcal{W}_2$ and $\mathcal{K}_2$ are also \fb\
as \iss.
\end{remark}

In a similar manner, Theorem~\ref{thm:main} allows one to solve the finite basis problem for many other species of infinite diagram monoids
in the setting of both plain and \iss. These applications of Theorem~\ref{thm:main} will be published in a separate paper, while here we
restrict ourselves to demonstrating another application of rather a different flavor.

Recall the classic \Rm\ construction \cite[see][Chapter~3, for details and for the explanantion of the role played by this construction in
the structure theory of semigroups]{CP61}. Let $\mathcal{G}=\langle G,\cdot\rangle$ be a semigroup, $0$ a symbol beyond $G$, and
$I,\Lambda$ non-empty sets. Given a $\Lambda\times I$-matrix $P=(p_{\lambda i})$ over $G\cup\{0\}$, we define a multiplication~$\cdot$ on
the set $(I\times G\times\Lambda)\cup\{0\}$ by the following rules:
\begin{gather*}
a\cdot 0=0\cdot a:=0\ \text{ for all } a\in (I\times G\times\Lambda)\cup \{0\},\\
(i,g,\lambda)\cdot(j,h,\mu):=\begin{cases}
(i,gp_{\lambda j}h,\mu)&\ \text{if}\ p_{\lambda j}\ne0,\\
0 &\ \text{if}\ p_{\lambda j}=0.
\end{cases}
\end{gather*}
Then $\langle(I\times G\times\Lambda)\cup \{0\},\cdot\rangle$ becomes a semigroup denoted by $\mathcal{M}^0(I,\mathcal{G},\Lambda;P)$ and
is called the \emph{\Rms\ over $\mathcal{G}$ with the sandwich matrix $P$}. For a \sgp\ $\mathcal{S}$, we let $\mathcal{S}^1$ stand for the
monoid obtained from $\mathcal{S}$ by adjoining a new identity element.

\begin{theorem}
\label{thm:Rees matrix} Let $\mathcal{G}=\langle G,\cdot\rangle$ be an abelian group and
$\mathcal{S}=\mathcal{M}^0(I,\mathcal{G},\Lambda;P)$ be a \Rms\ over $\mathcal{G}$. If the matrix $P$ has a submatrix of one of the forms
$\left(\begin{smallmatrix} a & b\\ c & 0\end{smallmatrix}\right)$ or $\left(\begin{smallmatrix} 0 & b\\ c & 0\end{smallmatrix}\right)$
where $a,b,c\in G$, or $\left(\begin{smallmatrix} e & e\\ e & d\end{smallmatrix}\right)$ where $e$ is the identity of $\mathcal{G}$ and
$d\in G$ is an element of infinite order, then the monoid $\mathcal{S}^1$ is \nfb.
\end{theorem}

\begin{proof}
Let $\mathcal{E}=\langle\{e\},\cdot\rangle$ be the trivial group, and let $\overline{P}=(\bar p_{\lambda i})$ be the $\Lambda\times
I$-matrix over $\{e,0\}$ obtained when each non-zero entry of $P$ gets substituted by $e$. Consider the \Rms\
$\mathcal{T}=\mathcal{M}^0(I,\mathcal{E},\Lambda;\overline{P})$. It is easy to see that the map $\varphi$ defined by
$$1\mapsto 1,\ 0\mapsto 0,\ (i,g,\lambda)\mapsto(i,e,\lambda)$$
is a homomorphism from $\mathcal{S}^1$ onto $\mathcal{T}^1$. It is known \cite[see, e.g.,][proof of Theorem 3.3]{Ha91} that every \Rms\
over $\mathcal{E}$ belongs to the variety generated by the 5-element \sgp\ $A_2$ that can be defined as the \Rms\ over $\mathcal{E}$ with
the sandwich matrix $\left(\begin{smallmatrix}e & e\\ e & 0\end{smallmatrix}\right)$. Therefore $\mathcal{T}^1$ lies in the variety $\var
A_2^1$. The inverse image of an arbitrary element $(i,e,\lambda)\in\mathcal{T}$ under $\varphi$ consists of all triples of the form
$(i,g,\lambda)$ where $g$ runs over $G$. If for some $j\in I, \mu\in\Lambda$, the triple $(j,e,\mu)$ is an idempotent in $\mathcal{T}$,
then $\bar p_{\mu j}\ne 0$ whence $p_{\mu j}\ne 0$ as well. Therefore the product of any two triples
$(j,g,\mu),(j,h,\mu)\in(j,e,\mu)\varphi^{-1}$ is equal to $(j,gp_{\mu j}h,\mu)$ and this result does not depend on the order of the factors
since the group $\mathcal{G}$ is abelian. Taking into account that $0\varphi^{-1}=\{0\}$ and $1\varphi^{-1}=\{1\}$, we see that the inverse
image under $\varphi$ of every idempotent in $\mathcal{T}^1$ is a commutative subsemigroup in $\mathcal{S}^1$. Thus, $\mathcal{S}^1$
belongs to the Mal'cev product $\mathbf{Com}\malcev\var A_2^1 $, and $\var A_2^1$ is locally finite as the variety generated by a finite
algebra \cite[see][Theorem~10.16]{BS81}.

In view of Theorem~\ref{thm:main}, it remains to verify that each Zimin word is an isoterm relative to $\mathcal{S}^1$. Here we invoke the
premise that the matrix $P$ has a $2\times 2$-submatrix of a specific form. We fix such a submatrix $P'$ of one of the given forms and let
$\Lambda'=\{\lambda,\mu\}\subseteq\Lambda$ and $I'=\{i,j\}\subseteq I$ be such that $P'$ occurs at the intersection of the rows
whose indices are in $\Lambda'$ with the columns whose indices are in $I'$.

First consider the case when $P'$ is either $\left(\begin{smallmatrix} a & b\\ c & 0\end{smallmatrix}\right)$ or $\left(\begin{smallmatrix} 0 & b\\
c & 0\end{smallmatrix}\right)$. Clearly, the \Rms\ $\mathcal{U}=\mathcal{M}^0(I',\mathcal{G},\Lambda';P')$ is a sub\sgp\ of $\mathcal{S}$
whence $\mathcal{U}^1$ is a subsemigroup of $\mathcal{S}^1$. Then the image of $\mathcal{U}^1$ under the homomorphism $\varphi$ is a
subsemigroup $\mathcal{V}^1$ of $\mathcal{T}^1$ where $\mathcal{V}$ can be identified with the \Rms\ over $\mathcal{E}$ whose sandwich
matrix is either $\left(\begin{smallmatrix}0 & e\\ e & 0\end{smallmatrix}\right)$ or $\left(\begin{smallmatrix}e & e\\ e &
0\end{smallmatrix}\right)$. In the latter case, the \sgp\ $\mathcal{V}$ is isomorphic to the \sgp\ $A_2$. We have already used the fact
that every \Rms\ over $\mathcal{E}$ belongs to the variety $\var A_2$; this implies that in any case the \Rms\
$B=\mathcal{M}^0(I',\mathcal{E},\Lambda';\left(\begin{smallmatrix}0 & e\\ e & 0\end{smallmatrix}\right))$ belongs to the variety
$\var\mathcal{V}$. Hence $B^1\in\var\mathcal{V}^1$, and it is easy to verify that the bijection
$$1{\mapsto}1,\ 0{\mapsto}0,\ (i,e,\lambda){\mapsto}a,\ (j,e,\mu){\mapsto}b,\ (i,e,\mu){\mapsto}ab,\
(j,e,\lambda){\mapsto}ba$$ is an isomorphism between $B^1$ and the 6-element Brandt monoid $B_2^1$ (we have defined the latter monoid
in the proof of Theorem~\ref{thm:kauffman}). Thus, $B_2^1$ lies in the variety $\var\mathcal{S}^1$, and each Zimin word is an isoterm
relative to $B_2^1$ \citep[see][Lemma 3.7]{Sa87}.

Now suppose that $P'=\left(\begin{smallmatrix} e & e\\ e & d\end{smallmatrix}\right)$ with $d\in G$ being an element of infinite order. One
readily verifies that the set
$$R=\bigl\{(k,d^n,\nu)\mid k\in I',\ \nu\in \Lambda',\ n=0,1,2,\dots\bigr\}$$
forms a subsemigroup in $\mathcal{S}$ while the set
$$J=\bigl\{(k,d^n,\nu)\mid k\in I',\ \nu\in \Lambda',\ n=1,2,\dots\bigr\}$$
forms an ideal in $R$. It is easy to calculate that the Rees quotient $R/J$ is isomorphic to the \sgp\ $A_2$, and we again conclude that
$B_2^1$ lies in the variety $\var\mathcal{S}^1$.
\end{proof}

\begin{remark}
Suppose that $\mathcal{G}=\langle G,\cdot\rangle$ in an abelian group, $I$ is a non-empty set, $0$ is a symbol beyond $G$, and $P=(p_{ij})$
is a symmetric $I\times I$-matrix over $G\cup\{0\}$. Then one can equip the \Rms\ $\mathcal{M}^0(I,\mathcal{G},I;P)$ with an involution by
letting $0^\star:=0,\quad (i,g,j)^\star:=(j,g,i)$. A version of Theorem~\ref{thm:Rees matrix} holds also for involution monoids that are
obtained from such \iss\ by adjoining a new identity element.
\end{remark}

\begin{remark}
Theorem \ref{thm:Rees matrix} remains valid if we replace the abelian group $\mathcal{G}$ by an arbitrary semigroup
$\mathcal{H}$ from a variety $\mathbf{U}$ all of whose periodic members are locally finite. In the matrix
$\left(\begin{smallmatrix} e & e\\ e & d\end{smallmatrix}\right)$ the elements $e,d\in\mathcal{H}$ have to be chosen such that
$e^2=e$, $ed=d=de$ and $d^n\ne e$ for all positive integers $n$.
\end{remark}

\begin{remark}
Readers familiar with the role of Rees matrix semigroups in the structure theory of semigroups will notice that Theorem \ref{thm:Rees
matrix} shows that for each \emph{completely simple} semigroup $\mathcal{S}$ which admits two idempotents whose product has infinite order
and whose maximal subgroups are abelian, the monoid $\mathcal{S}^1$ is \nfb. Indeed, $\mathcal{S}$ admits a Rees matrix representation
$\mathcal{M}(I,\mathcal{G},\Lambda;P)$ (the construction mentioned above but without $0$) such that $P$ has a submatrix of the form
$\left(\begin{smallmatrix} e & e\\ e & d\end{smallmatrix}\right)$ and $d$ has infinite order in $\mathcal{G}$. The proof of Theorem
\ref{thm:Rees matrix} then shows that $\mathcal{S}^1\in \var(\mathbf{Com}\malcev\mathbf{B})$ and $A_2^1\in \var\mathcal{S}^1$ hence each
Zimin word is an isoterm relative to $\mathcal{S}^1$.
\end{remark}

\small

\noindent\textbf{Acknowledgements.} Yuzhu Chen, Xun Hu, Yanfeng Luo have been partially supported by the Natural Science Foundation of
China (projects no.\ 10971086, 11371177). M. V. Volkov acknowledges support from the Presidential Programme ``Leading Scientific Schools of
the Russian Federation'', project no.\ 5161.2014.1, and from the Russian Foundation for Basic Research, project no.\ 14-01-00524.

\end{document}